\documentclass[12pt]{amsart}
\usepackage{latexsym,enumerate}
\usepackage{amssymb, xcolor,mathrsfs}
\usepackage[colorlinks]{hyperref}

\newtheorem{theorem}{Theorem}[section]
\newtheorem{lemma}[theorem]{Lemma}
\newtheorem{proposition}[theorem]{Proposition}
\newtheorem{corollary}[theorem]{Corollary}

\theoremstyle{definition}

\theoremstyle{remark}

\numberwithin{equation}{section}

\newcommand{\GL}{{\mathrm {GL}}}
\newcommand{\PGL}{{\mathrm {PGL}}}
\newcommand{\SL}{{\mathrm {SL}}}
\newcommand{\PSL}{{\mathrm {PSL}}}
\newcommand{\POmega}{{\mathrm {P\Omega}}}

\newcommand{\GU}{{\mathrm {GU}}}
\newcommand{\PGU}{{\mathrm {PGU}}}

\newcommand{\PSU}{{\mathrm {PSU}}}

\newcommand{\GO}{{\mathrm {GO}}}

\newcommand{\SO}{{\mathrm {SO}}}

\newcommand{\Sp}{{\mathrm {Sp}}}
\newcommand{\PSp}{{\mathrm {PSp}}}
\newcommand{\PCSp}{{\mathrm {PCSp}}}

\newcommand{\Aut}{{\mathrm {Aut}}}
\newcommand{\Out}{{\mathrm {Out}}}

\newcommand{\Irr}{{\mathrm {Irr}}}

\newcommand{\Syl}{{\mathrm {Syl}}}

\newcommand{\ZZ}{{\mathbb Z}}
\newcommand{\NN}{{\mathbb N}}

\newcommand{\GG}{{\mathbb G}}

\newcommand{\FF}{{\mathbb F}}

\newcommand{\ta}{\hspace{0.5mm}^{2}\hspace*{-0.2mm}}
\newcommand{\tb}{\hspace{0.5mm}^{3}\hspace*{-0.2mm}}

\newcommand{\wt}{\widetilde}

\newcommand{\bC}{{\mathbf{C}}}
\newcommand{\bG}{{\mathbf{G}}}
\newcommand{\bS}{{\mathbf{S}}}
\newcommand{\bH}{{\mathbf{H}}}
\newcommand{\bT}{{\mathbf{T}}}
\newcommand{\bL}{{\mathbf{L}}}
\newcommand{\bO}{{\mathbf{O}}}
\newcommand{\bN}{{\mathbf{N}}}
\newcommand{\bZ}{{\mathbf{Z}}}
\newcommand{\Al}{\textup{\textsf{A}}}
\newcommand{\Sy}{\textup{\textsf{S}}}
\newcommand{\tw}[1]{{}^#1}

\newcommand{\pcore}{\mathbf{O}}

\begin{document}

\title[On H\'{e}thelyi-K\"{u}lshammer's conjecture for principal blocks]
{On H\'{e}thelyi-K\"{u}lshammer's conjecture\\ for principal blocks}

\author[N.\,N. Hung]{Nguyen Ngoc Hung}
\address{Department of Mathematics, The University of Akron, Akron,
OH 44325, USA} \email{hungnguyen@uakron.edu}

\author[A. A. Schaeffer Fry]{A. A. Schaeffer Fry}
\address{Department of Mathematics and Statistics, Metropolitan State University of Denver, Denver, CO 80217, USA}
\email{aschaef6@msudenver.edu}

\thanks{The first author thanks Gunter Malle for several stimulating discussions on the relation
between the relative Weyl groups of $e$-split Levi subgroups and
maximal tori.
The second author is partially supported by the National
Science Foundation under Grant No. DMS-1801156.}




\begin{abstract} We prove that the number of irreducible ordinary
characters in the principal $p$-block of a finite group $G$ of order
divisible by $p$ is always at least $2\sqrt{p-1}$. This confirms a
conjecture of H\'{e}thelyi and K\"{u}lshammer
\cite{Hethelyi-kulshammer} for principal blocks and provides an
affirmative answer to Brauer's Problem 21 \cite{Brauer63} for
principal blocks of bounded defect. Our proof relies on recent works
of Mar\'{o}ti \cite{Maroti16} and Malle-Mar\'{o}ti
\cite{Malle-Maroti} on bounding the conjugacy class number and the
number of $p'$-degree irreducible characters of finite groups,
earlier works of Brou\'{e}-Malle-Michel \cite{Broue-Malle-Michel}
and Cabanes-Enguehard \cite{Cabanes-Enguehard94,Cabanes-book} on 
the distribution of characters into unipotent blocks and
$e$-Harish-Chandra series of finite reductive groups, and known
cases of the Alperin-McKay conjecture.
\end{abstract}

\maketitle


\section{Introduction}

Bounding the number $k(G)$ of conjugacy classes of a finite group
$G$ in terms of a certain invariant associated to $G$ is a
fundamental problem in group representation theory. Let $p$ be a
prime dividing the order of $G$. As observed by Pyber, a result of
Brauer \cite{Brauer42} on groups $G$ of order divisible by $p$ but
not by $p^{2}$ implies that $k(G) \geq 2 \sqrt{p-1}$ for those
groups, and the bound was later conjectured to be true
for all finite groups. After several earlier partial results
\cite{Hethelyi-kulshammer,HK,Malle06,Keller09,Hethelyietal}, the
conjecture was finally proved by Mar\'{o}ti \cite{Maroti16}.

An equally important problem in modular representation theory is to
bound the number $k(B)$ of ordinary irreducible characters in a
block. It is not surprising that this problem is closely related to
the previous one on bounding $k(G)$; for instance, the $p$-solvable
case of the celebrated Brauer's $k(B)$-conjecture \cite[Problem
20]{Brauer63}, which asserts that $k(B)$ is bounded \emph{above} by
the order of a defect group for $B$, was known to be equivalent to
the coprime $k(GV)$-problem, which in turn was eventually solved in
2004 \cite{Glucketal,Schmid}. While there have been a number of
results on upper bounds for $k(B)$
\cite{Brauer-Feit,Robinson,Sambale17,malle17}, not much has been
done on lower bounds.

In the proof of $k(G) \geq 2 \sqrt{p-1}$ for solvable groups,
H\'{e}thelyi and K\"{u}lshammer
\cite[Remark~(ii)]{Hethelyi-kulshammer} speculated that
``\emph{perhaps it is even true that $k(B) \geq 2\sqrt{p -1}$ for
every $p$-block $B$ of positive defect, where $k(B)$ denotes the
number of irreducible ordinary characters in $B$}''. Of course,
H\'{e}thelyi and K\"{u}lshammer were aware of blocks of defect zero,
which have a unique irreducible ordinary character (whose degree has
the same $p$-part as the order of the group) and a unique
irreducible Brauer character as well, see \cite[Theorem
3.18]{Navarro}.

The main aim of this paper is to confirm
H\'{e}thelyi-K\"{u}lshammer's conjecture for principal blocks.

\begin{theorem}\label{main-theorem}
Let $G$ be a finite group and $p$ a prime such that $p\mid |G|$. Let
$B_0(G)$ denote the principal $p$-block of $G$. Then $k(B_0(G))\geq
2\sqrt{p-1}$.
\end{theorem}

Problem 21 in Brauer's famous list \cite{Brauer63} asks whether
there exists a function $f(q)$ on prime powers $q$ such that
$f(q)\rightarrow \infty$ for $q\rightarrow \infty$ and that
$k(B)\geq f(p^{d(B)})$ for every $p$-block $B$ of defect $d(B)>0$.
Our Theorem~\ref{main-theorem} provides an affirmative answer to
this question for principal blocks of \emph{bounded} defect.

Building upon the ideas in \cite{Maroti16} and the subsequent paper
\cite{Malle-Maroti} of Malle and Mar\'{o}ti on bounding the number
of $p'$-degree irreducible characters in a finite group, we observe
that H\'{e}thelyi-K\"{u}lshammer's conjecture for principal blocks
is essentially a problem on bounding the number of irreducible
ordinary characters in principal blocks of almost simple groups, as
well as bounding the number of orbits of irreducible characters in
principal blocks of simple groups under the action of their
automorphism groups.

\begin{theorem}\label{main-thm-simple} Let $S$ be a non-abelian simple group and $p$ a
prime such that $p\mid |S|$. Let $G$ be an almost simple group with
socle $S$ such that $p\nmid |G/S|$. Then
\begin{itemize}
\item[(i)] $k(B_0(G))\geq 2\sqrt{p-1}$. Moreover, $k(B_0(G))>
2\sqrt{p-1}$ if $S$ does not have cyclic Sylow $p$-subgroups.

\item[(ii)] Assume furthermore that $p\geq 11$ and $S$ does not have cyclic Sylow $p$-subgroups. Then
the number of $\Aut(S)$-orbits on $\Irr(B_0(S))$ is at least
$2(p-1)^{1/4}$.
\end{itemize}
\end{theorem}

As we will explain in the next section, Theorem~\ref{main-theorem}
is a consequence of \cite{Maroti16} and the well-known Alperin-McKay
conjecture, which asserts that the number of irreducible characters
of height 0 in a block $B$ of a finite group $G$ coincides with the
number of irreducible characters of height 0 in the Brauer
correspondent of $B$ of the normalizer of a defect subgroup for $B$
in $G$. We take advantage of the recent advances on the conjecture
in the proof of our results, particularly the fact that Sp\"{a}th's
inductive Alperin-McKay conditions hold for all $p$-blocks with
cyclic defect groups \cite{Spath,KS16}. This explains why simple
groups with cyclic Sylow $p$-subgroups are excluded in
Theorem~\ref{main-thm-simple}(ii). Additionally, we take
advantage of recent results on the possible structure of defect
groups of principal blocks with few ordinary characters
\cite{KoshitanikB3,RSV21}, and this explains why the smaller values
of $p$ are excluded in Theorem~\ref{main-thm-simple}(ii).

Theorem \ref{main-thm-simple} turns out to be straightforward for
alternating groups or groups of Lie type in characteristic $p$, but
highly nontrivial for groups of Lie type in characteristic not equal
to $p$. We make use of Cabanes-Enguehard's results on unipotent
blocks \cite{Cabanes-book} to prove that the so-called semisimple
characters of $S$ all fall into the principal $p$-block $B_0(S)$ in
a certain nice situation. This and results of Brou\'{e}-Malle-Michel
\cite{Broue-Malle-Michel} and Cabanes-Enguehard
\cite{Cabanes-Enguehard94} on the compatibility between the
distributions of unipotent characters into unipotent blocks and
$e$-Harish-Chandra series allow us to obtain a general bound for the
number of $\Aut(S)$-orbits of characters in $\Irr(B_0(S))$ in terms
of certain data associated to $S$, for $S$  a simple group of
Lie type, see Theorem~\ref{theorem-bound}. We hope this result will
be useful in other purposes.

The next result classifies groups for which $k(B_0(G))$ is minimal
in the sense of Theorem~\ref{main-theorem}.

\begin{theorem}\label{theorem:equality}
Let $G$ be a finite group and $p$ a prime. Let $P$ be a Sylow
$p$-subgroup of $G$ and $B_0$ denote the principal $p$-block of $G$.
Then $k(B_0)=2\sqrt{p-1}$ if and only if $\sqrt{p-1}\in\NN$ and
$\bN_G(P)/\bO_{p'}(\bN_G(P))$ is isomorphic to the Frobenius group
$C_p\rtimes C_{\sqrt{p-1}}$.
\end{theorem}

We remark that, in the situation of Theorem~\ref{theorem:equality},
the number of $p'$-degree irreducible characters in $B_0(G)$ is also
equal to $2\sqrt{p-1}$. In general, if a $p$-block $B$ of a finite
group has an abelian defect group, then every ordinary irreducible
character of $B$ has height zero. This is the `if direction' of
Brauer's height-zero conjecture, which is now known to be true,
thanks to the work of Kessar and Malle \cite{Kessar-Malle}.
Theorem~\ref{main-theorem} therefore implies that if $P\in\Syl_p(G)$
is abelian and nontrivial then $k_0(B_0(G))\geq 2\sqrt{p-1}$, where
$k_0(B)$ denotes the number of height zero ordinary irreducible
characters of a block $B$.

Theorems~\ref{main-theorem} and \ref{theorem:equality} are useful in
the study of principal blocks with few height zero ordinary
irreducible characters. In fact, using them, we are able to show in
a forthcoming paper \cite{Hung-SF21} that $k_0(B_0(G))=3$ if and
only if $P\cong C_3$, and that $k_0(B_0(G))=4$ if and only if
$|P/P'|=4$ or $P\cong C_5$ and $\bN_G(P)/\bO_{p'}(\bN_G(P))$ is
isomorphic to the dihedral group $D_{10}$. These results have been
known only in the case $p\leq 3$, see \cite[Theorems~A and
C]{NavSamTie18}.

The paper is organized as follows. In Section~\ref{sec:first-obser},
we recall some known results on the Alperin-Mckay conjecture and
prove that our results follow when all the non-abelian composition
factors of $G$ have cyclic Sylow $p$-subgroups. We also prove
Theorem~\ref{main-thm-simple} for the sporadic simple groups and
groups of Lie type defined in characteristic $p$ in
Section~\ref{sec:first-obser}. The alternating groups are treated in
Section~\ref{sec:alt}. Section~\ref{sec:nonabelian} takes care of
the case when the Sylow $p$-subgroups of $S$ are non-abelian.
Sections~\ref{sec:linear-unitary}, \ref{sec:symp-orthog}, and
\ref{sec:proof-1.2-except} are devoted to proving
Theorem~\ref{main-thm-simple} for simple groups of Lie type defined in characteristics different from $p$. To do
so, in Section~\ref{sec:semi-chars}, we provide some background on
semisimple characters of finite reductive groups and prove that
those associated to conjugacy classes of $p$-elements belong to the
principal $p$-block in a certain situation, and in
Section~\ref{sec:gen-bound} we obtain a bound for the number of
$\Aut(S)$-orbits of characters in $\Irr(B_0(S))$. Finally, we finish
the proofs of Theorems~\ref{main-theorem} and \ref{theorem:equality}
in Section~\ref{sec:proof1.1}.


\section{Some first observations}\label{sec:first-obser}

In this section we make some observations toward the proofs of the
main results.

\subsection{The Alperin-McKay conjecture}\label{sec:AM}
The well-known Alperin-McKay (AM) conjecture predicts that the
number of irreducible characters of height zero in a block $B$ of a
finite group $G$ coincides with the number of irreducible characters
of height zero in the Brauer correspondent of $B$ of the normalizer of
a defect subgroup of $B$ in $G$. For the principal blocks, the
conjecture is equivalent to
\[
k_{p'}(B_0(G))=k_{p'}(B_0(\bN_G(P))),
\]
where $P$ is a Sylow $p$-subgroup of $G$ and $k_{p'}(B_0(G))$
denotes the number of $p'$-degree irreducible ordinary characters in
$B_0(G)$.

On the other hand, if $p\mid |G|$, we have
\begin{align*}
k_{p'}(B_0(\bN_G(P)))&\geq k_{p'}(B_0(\bN_G(P)/P'))\\
&= k\left(B_0(\bN_G(P)/P')\right)\\
&=k\left(B_0\left((\bN_G(P)/P')/\pcore_{p'}(\bN_G(P)/P')\right)\right)\\
&= k\left((\bN_G(P)/P')/\pcore_{p'}(\bN_G(P)/P')\right)\\
&\geq 2\sqrt{p-1},
\end{align*}
where the first inequality follows from \cite[p. 137]{Navarro}, the
first equality follows from the fact that every irreducible ordinary
character of $\bN_G(P)/P'$ has $p'$-degree, the last two equalities
follow from \cite[Theorem 9.9]{Navarro} and Fong's theorem (see
\cite[Theorem 10.20]{Navarro}), and the last inequality follows from
\cite{Maroti16}. Therefore, if the AM conjecture holds for $G$ and
$p$, then the number of $p'$-degree irreducible ordinary characters
in $B_0(G)$ is bounded by $2\sqrt{p-1}$.

From this, we see that Theorems \ref{main-theorem} and
\ref{main-thm-simple}(i) hold if the AM conjecture holds for
$(G,p)$. We now prove that the same is true for
Theorem~\ref{theorem:equality}. Note that the ``if" implication of this
theorem is clear. Assume that the AM conjecture holds for $B_0(G)$
and $k(B_0(G))=2\sqrt{p-1}$ for some prime $p$ such that
$\sqrt{p-1}\in\NN$. Then, as seen above, we have
\[2\sqrt{p-1}=k(B_0(G))\geq
k((\bN_G(P)/P')/\pcore_{p'}(\bN_G(P)/P'))\geq2\sqrt{p-1},\] implying
\[
k((\bN_G(P)/P')/\pcore_{p'}(\bN_G(P)/P'))=2\sqrt{p-1},
\]
and thus $(\bN_G(P)/P')/\pcore_{p'}(\bN_G(P)/P')$ is isomorphic to
the Frobenius group $C_p\rtimes C_{\sqrt{p-1}}$, by \cite[Theorem
1]{Maroti16}. In particular, $P/P'\cong C_p$, implying that $P\cong
C_p$, and hence it follows that $\bN_G(P)/\bO_{p'}(\bN_G(P))$ is
isomorphic to the Frobenius group $C_p\rtimes C_{\sqrt{p-1}}$, as
wanted.

The AM conjecture is known to be true when $G$ has a cyclic Sylow
$p$-subgroup by Dade's theory \cite{Dade66}. In fact, by
\cite{Spath,KS16}, the so-called \emph{inductive Alperin-McKay
conditions} are satisfied for all blocks with cyclic defect groups.
Therefore, we have:

\begin{lemma}[Koshitani-Sp\"{a}th]\label{lemma-KS}
Let $p$ be a prime. Assume that all the non-abelian composition
factors of a finite group $G$ have cyclic Sylow $p$-subgroups. Then
the Alperin-Mckay conjecture holds for $G$ and $p$, and thus
Theorems~\ref{main-theorem}, \ref{main-thm-simple}(i), and
\ref{theorem:equality} hold for $G$ and $p$.
\end{lemma}

Note that the linear groups $\PSL_2(q)$, the Suzuki groups $\ta
B_2(2^{2f+1})$ and the Ree groups $\ta G_2(3^{2f+1})$ all have
cyclic Sylow $p$-subgroups for $p$ different from the defining
characteristic of the group. So Theorem~\ref{main-thm-simple}
automatically follows from Lemma~\ref{lemma-KS} for these groups in
characteristic not equal to $p$.

\subsection{Small blocks} Blocks with a small number of ordinary characters have
been studied significantly in the literature. In particular, the
possible structure of defect groups of principal blocks with at most
5 ordinary irreducible characters are now known, see
\cite{Brandt82,Belonogov,KoshitanikB3,RSV21}. Using these results,
we can easily confirm our results for $p\leq 7$. For instance, to
prove Theorems~\ref{main-theorem} and \ref{main-thm-simple} for
$p=7$ it is enough to assume that $k(B_0(G))\leq 4$, but by going
through the list of possible defect groups of $B_0(G)$, we then have
$\Syl_p(G)\in\{1,C_2,C_3,C_2\times C_2,C_4,C_5\}$, which cannot
happen. To prove Theorem~\ref{theorem:equality} for $p< 7$ we
note that if $p=5$ and $k(B_0(G))=4$ then $P=C_5$; and if $p=2$ and
$k(B_0(G))=2$ then $P=C_2$, both of which cases $P$ is cyclic, and
thus the result of Subsection~\ref{sec:AM} applies.

Therefore we will assume from now on that $p\geq 11$, unless stated
otherwise.

\subsection{Sporadic and the Tits groups} We remark that Theorem
\ref{main-thm-simple} can be confirmed directly using
\cite{Atl1,atlas2} or \cite{GAP48} for sporadic simple groups and
the Tits group. Therefore, we are left with the alternating groups
and groups of Lie type, which will be treated in the subsequent
sections.

\subsection{Groups of Lie type in characteristic $p$}\label{sec:defining}
 Let $S$ be a
simple group of Lie type defined over the field of $q=p^f$ elements.
According to results of Dagger and Humphreys on defect groups of
finite reductive groups in defining characteristic, see
\cite[Proposition 1.18 and Theorem 3.3]{Cabanes18} for instance, $S$
has only two $p$-blocks.  Namely, the only non-principal block is a defect-zero block containing only the Steinberg character of $S$. Therefore,
\[
k(B_0(S))= k(S)-1.
\]

Let $\bG$ be a simple algebraic group of simply connected type and let $F$ be a Steinberg endomorphism on $\bG$ such that $S=X/\bZ(X)$, where $X=\bG^F$.  Assume that the rank of $\bG$ is $r$. By a result of
Steinberg (see \cite[Theorem 3.1]{Fulman-Guralnick12}), $X$ has at
least $q^r$ semisimple conjugacy classes, and thus $k(X)>q^r$. It
follows that
\[
k(B_0(S))\geq \left[ \frac{q^r}{|\bZ(X)|}-1\right],
\]
which yields $k(B_0(S))\geq q^r/|\bZ(X)|$. Using the values of
$|\bZ(X)|$ and $|\Out(S)|$ available in \cite[p. xvi]{Atl1}, it is
straightforward to check that $q^r/|\bZ(X)|\geq
2\sqrt{p-1}|\Out(S)|$, proving Theorem~\ref{main-thm-simple} for the
relevant $S$ and $p$.


\section{Alternating groups}\label{sec:alt}

In this section we prove Theorem \ref{main-thm-simple} for the
alternating groups. The background on block theory of symmetric and
alternating groups can be found in \cite{Olsson93} for instance.

The ordinary irreducible characters of $\Sy_n$ are naturally labeled
by partitions of $n$. Two characters are in the same $p$-block if
and only if their corresponding partitions have the same $p$-cores,
which are obtained from the partitions by successive removals of rim
$p$-hooks until no $p$-hook is left. Therefore, $p$-blocks of
$\Sy_n$ are in one-to-one correspondence with $p$-cores of
partitions of $n$.

Let $B$ be a $p$-block of $\Sy_n$. The number $k(B)$ of ordinary
irreducible characters in $B$ turns out to depend only on $p$
and the so-called \emph{weight} of $B$, which is defined to be
$w(B):=(n-|\mu|)/p$, where $\mu$ is the $p$-core corresponding to
$B$ under the aforementioned correspondence. In fact,
\[
k(B)=k(p,w(B)):=\Sigma_{(w_0,w_1...,w_{p-1})}\pi(w_0)\pi(w_1)\cdots
\pi(w_{p-1}),
\]
where $(w_0,w_1...,w_{p-1})$ runs through all $p$-tuples of
non-negative integers such that $w(B)=\Sigma_{i=0}^{p-1} w_i$ and
$\pi(x)$ is the number of partitions of $x$, see \cite[Proposition
11.4]{Olsson93}. Note that $k(p,w(B))$ is precisely the number of
$p$-tuples of partitions of $w(B)$.

For the principal block $B_0(\Sy_n)$ of $\Sy_n$, we have
$w(B_0(\Sy_n))=[n/p]$, which is at least 1 by the assumption $p\mid
|S|$. It follows that
\[
k(B_0(\Sy_n))\geq k(p,1)=p\geq 2\sqrt{p-1}.
\]
Moreover, according to \cite[Proposition 2.8]{Olsson92}, when $p$ is
odd and $\widetilde{B}$ is a block of $\Al_n$ covered by $B$, then
$B$ and $\widetilde{B}$ have the same number of irreducible ordinary
characters (and indeed the same number of irreducible Brauer
characters as well). In particular, when p is odd, we have
$k(B_0(\Al_n))=k(B_0(\Sy_n))\geq 2\sqrt{p-1}$, which proves
Theorem~\ref{main-thm-simple}(i) for the alternating groups.

For part (ii) of Theorem \ref{main-thm-simple}, recall that $p\geq 11$, and thus $n\geq 11$ and $\Aut(S)=\Sy_n$. The
number of $\Sy_n$-orbits on $\Irr(B_0(\Al_n))$ is at least
$k(B_0(\Al_n))/2$, which in turn is at least
$1+(p-1)/2=(p+1)/2>2(p-1)^{1/4}$, and this proves
Theorem~\ref{main-thm-simple}(ii) for the alternating groups.


\section{Groups of Lie type: the non-abelian $p$-Sylow case}\label{sec:nonabelian}

In this section, we let $\bG$ be a simple
algebraic group of adjoint type and $F$ a Steinberg endomorphism on
$\bG$ such that $S\cong [\GG,\GG]$ where $\GG:=\bG^F$. Let
$q=\ell^f$ with $\ell\neq p$ be the absolute value of all
eigenvalues of $F$ on the character group of an $F$-stable maximal
torus of $G$. Recall that we are assuming $p\geq 11$.

In this section we prove Theorem~\ref{main-thm-simple} for those $S$
of Lie type in characteristic different from $p$ such that the Sylow
$p$-subgroups of $\GG$ are non-abelian.

In that case, there are then more than one $d\in \NN$ such that
$p\mid \Phi_d(q)$ with $\Phi_d$ dividing the order polynomial of
$(\bG,F)$. Here, as usual, $\Phi_d$ denotes the $d$th cyclotomic
polynomial. (In fact, if there a unique such $d$, then a Sylow
$p$-subgroup of $\GG$ is contained in a Sylow $d$-torus of $\GG$,
and hence is abelian, see \cite[Theorem 25.14]{malletesterman}.)

Let $e_p(q)$ denote the multiplicative order of $q$ modulo $p$. Note
that, by \cite[Lemma 25.13]{malletesterman}, $p\mid \Phi_d(q)$ if
and only if $d=e_p(q)p^i$ for some $i\geq 0$. Therefore, as there is
more than one $d\in \NN$ such that $p\mid \Phi_d(q)$, we must have
$p\mid d$ for some $d\in \NN$ such that $\Phi_d$ divides the order
polynomial of $(\bG,F)$. The fact that $p\geq 11$ then rules out the
cases when $\bG$ is of exceptional type and thus we are left with
only the classical types. That is, $\GG=\PGL_n(q)$, $\PGU_n(q)$,
$\SO_{2n+1}(q)$, $\PCSp_{2n}(q)$, or
$\mathrm{P}(\mathrm{CO}^{\pm}_{2n}(q))^0$.

For $\GG=\PGL_n(q)$ or $\PGU_n(q)$, we define $e$ to be the smallest
positive integer such that $p\mid (q^e-(\epsilon)^e)$ ($\epsilon=1$
for linear groups and $\epsilon=-1$ for unitary groups), so that
$e=e_p(q)$ when $\GG=\PGL_n(q)$ or $\GG=\PGU_n(q)$ and $4\mid
e_p(q)$, $e=e_p(q)/2$ when $\GG=\PGU_n(q)$ and $2\mid e_p(q)$ but
$4\nmid e_p(q)$, and $e=2e_p(q)$ when $\GG=\PGU_n(q)$ and $2\nmid
e_p(q)$. For $\GG=\SO_{2n+1}(q)$, $\PCSp_{2n}(q)$, or
$\mathrm{P}(\mathrm{CO}^{\pm}_{2n}(q))^0$, we define $e$ to be the
smallest positive integer such that $p\mid (q^e\pm 1)$, so that
$e=e_p(q)$ when $e_p(q)$ is odd and $e=e_p(q)/2$ when $e_p(q)$ is
even.

Let $n=we+m$ where $0\leq m< e$. We claim that $p\leq w$. To see
this, first assume that $\GG=\PGL_n(q)$. Then, as mentioned above,
$ep\leq n$, which implies that $ep<(w+1)e$, and thus $p\leq w$.
Next, assume that $\GG=\SO_{2n+1}(q)$, $\PCSp_{2n}(q)$, or
$\mathrm{P}(\mathrm{CO}^{\pm}_{2n}(q))^0$. If $e=e_p(q)$ is odd,
then since $p\mid (q^e-1)$ and $\gcd(q^e-1,q^i+1)\leq 2$ for every
$i\in\NN$, we have $p\mid (q^j-1)$ for some $e<j\leq n$, and it
follows that $ep\leq n$, implying $p\leq w$. On the other hand, if
$2e=e_p(q)$ is even then $2ep=e_p(q)p\leq 2n<2(w+1)e$, which also
implies that $p\leq w$. Finally, assume $\GG=\PGU_n(q)$. The case
$4\mid e_p(q)$ is argued as in the case $S=\PGL_n(q)$; the case
$2\mid e_p(q)$ but $4\nmid e_p(q)$ is argued as in the case
$S=\SO_{2n+1}(q)$ and $2\mid e_p(q)$. For the last case $2\nmid
e_p(q)$, we have $ep/2=e_p(q)p$, and in order for $\Phi_{e_p(q)p}$
dividing the generic order of $|\PGU_n(q)|$, $e_p(q)p\leq n/2$, and
hence it follows that $ep\leq n$, which also implies that $p\leq w$.
The claim is fully proved.

By \cite[Theorem 3.2]{Broue-Malle-Michel} and \cite[Main
Theorem]{Cabanes-Enguehard94}, the number of unipotent characters of
$\GG$ in the principal block $B_0(\GG)$ is equal to $k(W_e)$ - the
number of irreducible complex characters of the relative Weyl group
$W_e$ of a Sylow $e_p(q)$-torus of $\GG$. This $W_e$ is the wreath
product $C_e\wr \Sy_w$ when $\bG$ is of type $A$ and is a subgroup
of index 1 or 2 of $C_{2e}\wr \Sy_w$ when $\bG$ is of type $B$, $C$,
or $D$, see \cite[Proposition 5.5 and its proof]{Malle-Maroti}. In
any case we have that the number of unipotent characters in
$\Irr(B_0(\GG))$ is at least $k(\Sy_w)/2=\pi(w)/2$, which in turns
is at least $\pi(p)/2$ as $p\leq w$. Since every unipotent character
of $\GG$ restricts irreducibly to $S$, it follows that the number of
unipotent characters in $\Irr(B_0(S))$ is at least $\pi(p)/2$.

By a result of Lusztig (see \cite[Theorem 2.5]{Malle08}), every
unipotent character of a simple group of Lie type lies in a
$\Aut(S)$-orbit of length at most 3. (In fact, every $\Aut(S)$-orbit
on unipotent characters of $S$ has length 1 or 2, except when
$S=P\Omega_8^+(q)$ whose the graph automorphism of order 3 produces
two orbits of length 3.) Therefore, together with the conclusion of
the previous paragraph, we deduce that the number of
$\Aut(S)$-orbits on $\Irr(B_0(S))$ is at least $\pi(p)/6$. This
bound is greater than $2\sqrt{p-1}$ when $p\geq 11$, as required.

\section{Semisimple characters and principal
blocks}\label{sec:semi-chars}

Before continuing with our proof of Theorem \ref{main-thm-simple}
for groups of Lie type, we recall some background on certain
characters known as \emph{semisimple characters} and show how they
fall into the principal block in a certain situation. Background on
character theory of finite reductive groups can be found in
\cite{Carter85,Cabanes-book,Digne-Michel91}. Let $\bG$ be a
connected reductive group defined over $\FF_q$ and $F$ an associated
Frobenius endomorphism on $\bG$. Let $\bG^\ast$ be an algebraic
group with a Frobenius endomorphism which, for simplicity, we denote
by the same $F$, such that $(\bG,F)$ is in duality to
$(\bG^\ast,F)$.

Let $t$ be a semisimple element of $(\bG^\ast)^F$. The rational
Lusztig series $\mathcal{E}(\bG^F,(t))$ associated to the
$(\bG^\ast)^F$-conjugacy class $(t)$ of $t$ is defined to be the set
of irreducible characters of $\bG^F$ occurring in some
Deligne-Lusztig character $R_{\bT}^{\bG}\theta$, where $\bT$ is an
$F$-stable maximal torus of $\bG$ and $\theta\in\Irr(\bT^F)$ such
that $(\bT,\theta)$ corresponds in duality to a pair $(\bT^\ast,s)$
with $s\in\bT^\ast \cap (t)$. Here we recall that there is a
one-to-one duality correspondence between $\bG^F$-conjugacy classes
of pairs $(\bT,\theta)$, where $\bT$ is an $F$-stable maximal torus
of $\bG$ and $\theta\in\Irr(\bT^F)$, and the
$(\bG^\ast)^F$-conjugacy classes of pairs $(\bT^\ast,s)$, where
$\bT^\ast$ is dual to $\bT$ and $s\in(\bT^\ast)^F$.

We continue to let $t$ be a semisimple element of $(\bG^\ast)^F$ and
assume furthermore that $\bC_{\bG^\ast}(t)$ is a Levi subgroup of
$\bG^\ast$. Let $\bG(t)$ be a Levi subgroup of $\bG$ in duality with
$\bC_{\bG^\ast}(t)$ and $\mathbf{P}$ be a parabolic subgroup of
$\bG$ for which $\bG(t)$ is the Levi complement. The twisted
induction $R_{\bG(t)\subseteq P}^\bG$ and the multiplication by
$\widehat{t}$, a certain linear character of $\Irr(\bG(t)^F)$
naturally defined by $t$ (see \cite[(8.19)]{Cabanes-book}), then
induce a bijection between the Lusztig series
$\mathcal{E}(\bG(t)^F,1)$ and $\mathcal{E}(\bG^F,(t))$, see
\cite[Proposition 8.26 and Theorem 8.27]{Cabanes-book} or
\cite[Theorem 13.25 and Proposition 13.30]{Digne-Michel91}. In fact,
for each $\lambda\in \mathcal{E}(\bG(t)^F,1)$, one has
\[\varepsilon_{\bG}\varepsilon_{\bG(t)}R_{\bG(t)\subseteq
P}^\bG(\widehat{t}\lambda)\in \mathcal{E}(\bG^F,(t)),\] where
$\varepsilon_{\bG}:=(-1)^{\sigma(\bG)}$ with $\sigma(\bG)$ the
$\mathbb{F}_q$-rank of $\bG$. Taking $\lambda$ to be trivial, we
have the character
\[\chi_{(t)}:=\varepsilon_{\bG}\varepsilon_{\bG(t)}R_{\bG(t)\subseteq
P}^\bG(\widehat{t}\mathbf{1}_{\bG(t)^F})\in\mathcal{E}(\bG^F,(t)),\]
which is often referred to as a semisimple character of $\bG^F$, of
degree
\[\chi_{(t)}(1)=|(\bG^\ast)^F:\bC_{{\bG^\ast}^F}(t)|_{\ell'},\] where
$\ell$ is the defining characteristic of $\bG$, see \cite[Theorem
13.23]{Digne-Michel91}.

By \cite[Theorem 9.12]{Cabanes-book}, every element of $B_0(\bG^F)$
lies in a Lusztig series $\mathcal{E}(\bG^F, (t))$ where $t$ is a
$p$-element of ${\bG^\ast}^F$. Hence one might ask which such $t$
indeed produce semisimple characters that contribute to the
principal block. We will see in the following theorem that in a
certain nice situation which is indeed enough for our purpose, the
centralizer $\bC_{\bG^\ast}(t)$ is a Levi subgroup of $\bG^\ast$,
and thus the semisimple character $\chi_{(t)}$ associated to $(t)$
is well-defined and belongs to $B_0(\bG^F)$.

In the following, we recall that a prime $p$ is good for $\bG$ if it
does not divide the coefficients of the highest root of the root
system associated to $\bG$.

\begin{theorem}\label{lemma-red-gp} Let $(\bG,F)$ be a connected reductive group defined
over $\FF_q$. Let $p$ be a good prime for $\bG$ and not dividing
$q$. Let $t$ be a $p$-element of ${\bG^\ast}^F$. If
$\bC_{\bG^\ast}(t)$ is connected, then the semisimple character
$\chi_{(t)}\in\Irr(\bG^F)$ belongs to the principal $p$-block of
$\bG^F$. In particular, if $\bZ(\bG)$ is connected, the character
$\chi_{(t)}$ belongs to the principal block of $\bG^F$ for every
$p$-element $t\in {\bG^\ast}^F$.
\end{theorem}

\begin{proof}
Assume that $\bC_{\bG^\ast}(t)$ is connected. Since $p$ is good for
$\bG$, $\bC_{\bG^\ast}(t)$ is then a Levi subgroup of $\bG^\ast$,
see \cite[Proposition 13.16]{Cabanes-book}. Define a Levi subgroup
$\bG(t)$ and parabolic subgroup $\mathbf{P}$ of $\bG$ as above.
Since
$\chi_{(t)}=\varepsilon_{\bG}\varepsilon_{\bG(t)}R_{\bG(t)\subseteq
\mathbf{P}}^\bG\left(\widehat{t}\mathbf{1}_{\bG(t)^F}\right)$,
\cite[Theorem 21.13]{Cabanes-book} implies that all the irreducible
constituents of $R_{\bG(t)\subseteq
\mathbf{P}}^\bG(\mathbf{1}_{\bG(t)^F})$ are unipotent characters of
$\bG^F$ which are in the same $p$-block as $\chi_{(t)}$. As the
trivial character is obviously a constituent of $R_{\bG(t)\subseteq
\mathbf{P}}^\bG(\mathbf{1}_{\bG(t)^F})$, we deduce that
$\chi_{(t)}\in B_0(\bG^F)$.

The second statement of the theorem immediately follows from the
first, as if $\bZ(\bG)$ is connected then the centralizer of every
semisimple element of $\bG^\ast$ is connected, see \cite[Lemma
13.14]{Digne-Michel91}.
\end{proof}


\section{Linear and Unitary Groups}\label{sec:linear-unitary}

In this section, we let $S=\PSL_n^\epsilon(q)$, where $p\nmid q$ and
$\epsilon\in\{\pm1\}$.  Here  $\PSL_n^\epsilon(q):=\PSL_n(q)$ in the
case $\epsilon=1$ and $\PSU_n(q)$ in the case $\epsilon=-1$, and
analogous for $\SL_n^\epsilon(q)$, $\GL_n^\epsilon(q)$, and
$\PGL_n^\epsilon(q)$.  We further let $\overline{q}:=q$ if
$\epsilon=1$ and $\overline{q}:=q^2$ if $\epsilon=-1$. Note that
with our notation, $\SL^\epsilon_n(q)$ and $\GL_n^\epsilon(q)$ are
naturally subgroups of $\SL_n(\overline{q})$ and
$\GL_n(\overline{q})$, respectively.

\begin{proposition}\label{prop:linear}
Let $S=\PSL^\epsilon_n(q)$ and let $p\nmid q$ be a prime. Then Theorem \ref{main-thm-simple} holds for any almost simple group $A$ with socle $S$ and $p\nmid |A/S|$.
\end{proposition}
\begin{proof}

With the results of the previous sections, we may assume $n\geq 3$,
$p\geq 11$, and $q$ is a power of some prime different from $p$.

Write $S=\PSL^\epsilon_n(q)$, $\GG=\PGL^\epsilon_n(q)$,
$\wt{G}=\GL^\epsilon_n(q)$, and $G=\SL^\epsilon_n(q)$.  Then we have
$G=[\wt{G}, \wt{G}]$, $S=G/\bZ(G)$, and $\GG=\wt{G}/\bZ(\wt{G})$.
From Section \ref{sec:nonabelian}, we may assume that Sylow
$p$-subgroups of $\GG$ are abelian, which implies that there is a
unique $e$ such that $p\mid \Phi_e({q})$ and $\Phi_e$ divides the
generic order polynomial of $\GG$. Here $e$ must be $e_p({q})$, the
multiplicative order of ${q}$ modulo $p$.  Note that this also
forces $p\nmid n$ by again appealing to \cite[Lemma
25.13]{malletesterman}.

We will further define $\overline{e}:=e_p(\overline{q})$ and $e'$ as
follows:
\[e':=\left\{\begin{array}{cc}
\overline{e} & \hbox{ if $\epsilon=1$ or if $\epsilon=-1$ and $p\mid q^{\overline{e}}-(-1)^{\overline{e}}$}\\
2\overline{e} & \hbox{ if $\epsilon=-1$ and $p\mid q^{\overline{e}}+(-1)^{\overline{e}}$}.\\
\end{array}\right.\]

To prove Theorem \ref{main-thm-simple}, our aim is to show that when
a Sylow $p$-subgroup of $S$ is not cyclic, then the number of
$\Aut(S)$-orbits on $\Irr(B_0(S))$ is larger than $2\sqrt{p-1}$.

Note that since $p\nmid \gcd(n,q-\epsilon)=|\bZ(G)|$, the
irreducible characters in the principal block of $S$ are the same as
that of $G$, under inflation (see \cite[Theorem 9.9]{Navarro}).
Similarly, if $e'>1$, then $p\nmid (q-\epsilon)=|\bZ(\wt{G})|$ and
an analogous statement holds for $\GG$ and $\wt{G}$. Hence, we begin
by studying $B_0(\wt{G})$, which will be sufficient for our purposes
in the case $e'>1$.

Let $n=we'+m$ with $0\leq m<e'$. Set
$p^a:=(\overline{q}^{\overline{e}}-1)_p\geq p$. The case $p\leq w$
was treated in Section \ref{sec:nonabelian}, so we assume that
$p>w$. Note that by \cite[Theorem 1.9]{Michler-Olsson},
$B_0(\wt{G})$ and $B_0(\GL^\epsilon_{we'}(q))$ have the same number
of ordinary irreducible characters, so we may assume that $n=we'$. (Note that the action of $\Aut(S)$ is analogous as well.)

Let $\mathcal{F}(p,a)$ denote the set of monic polynomials over
$\mathbb{F}_{\overline{q}}$ in the set $\mathscr{F}$ defined in
\cite{FS82} whose roots have $p$-power order in
$\overline{\FF}_{{q}}^{\times}$ at most $p^a$. Note that
$|\mathcal{F}(p,a)|=1+(p^a-1)/e'$, see \cite[p.
211]{Michler-Olsson}.

The conjugacy classes $(t):=t^{\wt{G}}$ of $p$-elements in $\wt{G}$
are parameterized by $p$-weight vectors of $w$, which are functions
$\mathbf{w}:=\mathbf{w}_{(t)}:\mathcal{F}(p,a)\rightarrow
\ZZ_{\geq0}$ such that $w=\sum_{g\in \mathcal{F}(p,a)}
\mathbf{w}(g)$. The characteristic polynomial of elements in $(t)$
is
\[
(x-1)^{e'\mathbf{w}(x-1)} \prod_{x-1\neq g\in \mathcal{F}(p,a)}
g^{\mathbf{w}(g)},
\]
and the centralizer of $t$ is
\[
\bC_{\wt{G}}(t)=\GL^\epsilon_{e'\mathbf{w}(x-1)}(q)\times
\prod_{x-1\neq g\in \mathcal{F}(p,a)}
\GL^\eta_{\mathbf{w}(g)}(q^{{e'}})
\]
 where $\eta=\epsilon$ unless $\epsilon=-1$ and $e'=2\overline{e}$, in which case $\eta=1$.

Each character in the Lusztig series $\mathcal{E}(\wt{G},t)$ is
labeled by $\chi_{t,\psi}$ where $\psi$ is a unipotent character of
$\bC_{\wt{G}}(t)$. So $\psi=\prod_{g\in\mathcal{F}(p,a)} \psi_g$
where $\psi_g$ is a unipotent character of
$\GL^\eta_{\mathbf{w}(g)}(q^{{e'}})$ if $g\neq x-1$ and of
$\GL_{e'\mathbf{w}(x-1)}(q)$ if $g=x-1$. Note that there is a
canonical correspondence between unipotent characters of
$\GL^\pm_x(q)$ and partitions of $x$, so we may view $\psi_g$ as a
partition of $\mathbf{w}(g)$ when $g\neq x-1$ and of
$e'\mathbf{w}(x-1)$ when $g=x-1$.
Further, by \cite[Theorem (7A)]{FS82}, 
the characters of $B_0(\wt{G})$ are exactly those $\chi_{t,\psi}$
satisfying $t$ is a $p$-element and the partition $\psi_{x-1}$ has
trivial $e'$-core.

By  \cite[Proposition 6]{Olsson84},
\[
k(B_0(\wt{G}))=k\left(e'+\frac{p^a-1}{e'},w\right),
\]
where $k(x,y)$ is as defined in Section \ref{sec:alt} above.  This
number is at least
\begin{equation}\label{eq:step1}
e'+\frac{p^a-1}{e'}\geq 2\sqrt{p^a-1}\geq 2\sqrt{p-1}.
\end{equation} But, recall that we wish to show that there are at least $2\sqrt{p-1}$ orbits on
$\Irr(B_0(S))$ under $\Aut(S)$.

Now, by taking $t=1$, the number of unipotent characters in
$B_0(\wt{G})$ is precisely $k(e',w)$. Note that $k(e',w)\geq
k(e',1)=e'$, and that further $k(e',w)\geq 2e'$ if $w\geq 2$ with strict inequality for $(e', w)\neq (1,2)$, and each unipotent character is $\Aut(S)$-invariant. So
we have at least $e'$ $\Aut(S)$-orbits of unipotent characters in
$B_0(\GG)$, and hence of $B_0(S)$, since restriction yields a
bijection between unipotent characters of $S$ and $\GG$.

Let $\wt{\bG}:=\GL_n(\overline{\mathbb{F}}_q)$ so that
$\wt{G}=\wt{\bG}^F$.  Since $\bZ(\wt{\bG})$ is connected,
\cite[Theorem 3.1]{CS13} yields that the ``Jordan decomposition"
$\psi_{t,\psi}\leftrightarrow(t, \psi)$ can be chosen to be
$\Aut(S)$-equivariant.  Since $\psi$ is a unipotent character of a
product of groups of the form $\GL^\pm_x(q^d)$, which are invariant
under automorphisms as discussed above, it follows that the orbit of
$\chi_{t,\psi}$ is completely determined by the action of $\Aut(S)$
on the class $(t)$.

Now, recall that the $\wt{G}$-class of $t$ is completely determined
by its eigenvalues.  Let  $|t|=p^c$ and note that $c\leq a$. By
viewing $t$ as an element  $1 \times \prod_{x-1\neq g\in
\mathcal{F}(p,a)} \zeta_g$ of \[\bZ(\bC_{\wt{G}}(t))\cong
C_{q-\epsilon} \times \prod_{x-1\neq g\in \mathcal{F}(p,a)}
C_{q^{e'}-\eta}\]
we see that for $\alpha\in\Aut(S)$, the eigenvalues of $t^\alpha$ are those of $t$ raised to some power $\eta q_0^e$ for some $\eta\in\{\pm1\}$ and some $q_0$ such that $\overline{q}$ is a power of $q_0$. 
This implies that the $\Aut(S)$-orbit of $(t)$ has size at most $\frac{p^c-1}{e'}\leq \frac{p^a-1}{e'}$.

Now, the Sylow $p$-subgroup $P$ of $\wt{G}$ is of the form
$C_{p^a}^w\leq (\mathbb{F}_{\overline{q}^{\overline{e}}}^\times)^w$.
Then if $w=1$, $P$ is cyclic, and hence we may assume that $w\geq
2$. In this case, we have at least
$\frac{1}{2}\left(\frac{p^a-1}{e'}\right)^2$ choices for $(t)\neq
(1)$, and hence at least
$\frac{1}{2}\left(\frac{p^a-1}{e'}\right)^2$ non-unipotent
characters in $B_0(\wt{G})$ by taking $\psi_{x-1}$ to be trivial.
This gives at least $\frac{p^a-1}{2e'}$ distinct orbits of
non-unipotent characters, and hence more than $2\sqrt{p-1}$ orbits
of characters in $B_0(\GG)$ under $\Aut(S)$ when $e'>1$, by \eqref{eq:step1} with $2e'$ rather than $e'$.
This completes the proof of Theorem \ref{main-thm-simple} for $S$ in
the case $e'>1$ by the discussion at the beginning of the section.

Finally suppose $e'=1$, so $w=n\geq 3$ and we may continue to assume
$p>w$. Consider the elements $t$ of $\wt{G}$ whose eigenvalues are
of the form $\{\zeta, \xi, (\zeta\xi)^{-1}, 1,\ldots, 1\}$ with
$\zeta$ and $\xi$ $p$-elements of
$C_{q-\epsilon}\leq\mathbb{F}_{\overline{q}}^\times$.   Note that
each member of $\mathcal{E}(\wt{G}, t)$ lies in the principal block
of $\wt{G}$ and that $t$ lies in $G=[\wt{G}, \wt{G}]$.  Further, $t$
cannot be conjugate to $tz$ for any nontrivial $z\in \bZ(\wt{G})$,
since such a $z$ would have determinant 1 and $p$-power order,
contradicting $p>n$. Then using \cite[Lemma 1.4]{RSV21} and
\cite[Proposition 2.6]{SFT21}, each character in such a
$\mathcal{E}(\wt{G}, t)$ is irreducible on restriction to $G$,
yielding at least $\frac{(p^a-1)^2}{2}$ non-unipotent members of $B_0(G)$.
Since the $\Aut(S)$-orbits of such characters are again of size at
most $p^a-1$, this yields at least $2+\frac{p^a-1}{2}$ distinct orbits, which is
larger than $2\sqrt{p-1}$.  This completes the proof
of Theorem \ref{main-thm-simple} in the case that
$S=\PSL^\epsilon_n(q)$.
\end{proof}


\section{Symplectic and Orthogonal Groups}\label{sec:symp-orthog}
In this section, we consider the simple groups coming from
orthogonal and symplectic groups.  That is, simple groups of Lie
type $B_n, C_n, D_n$, and $\tw{2}D_n$.  We let
$\epsilon\in\{\pm1\}$, and let $\POmega_{2n}^\epsilon(q)$ denote the
simple group of Lie type $D_n(q)$ for $\epsilon=1$ and of type
$\tw{2}D_n(q)$ for $\epsilon=-1$.

\begin{proposition}\label{prop:classical}
Let $q$ be a power of a prime different from $p$ and let $S=\PSp_{2n}(q)$ with $n\geq 2$, $\POmega_{2n+1}(q)$ with $n\geq 3$, or $\POmega^\epsilon_{2n}(q)$ with $n\geq 4$.  Then Theorem \ref{main-thm-simple} holds for any almost simple group $A$ with socle $S$ and $p\nmid |A/S|$.
\end{proposition}
\begin{proof}
With the results of the previous sections, we may again assume that
$p\geq 11$ and that a Sylow $p$-subgroup of $S$ is abelian, but not
cyclic.

Let $H$ be the corresponding symplectic or special orthogonal group
$\Sp_{2n}(q)$, $\SO_{2n+1}(q)$, or $\SO_{2n}^\epsilon(q)$ and let
$(\bH,F)$ be the corresponding simple algebraic group and Frobenius
endomorphism so that $H=\bH^F$. Let $G=\bG^F$ be the corresponding group
of simply connected type, so that $G=H$ in the symplectic case or
$G$ is the appropriate spin group in the orthogonal cases. Further,
let $(\bH^\ast, F)$ and $(\bG^\ast, F)$ be dual to $(\bH, F)$ and
$(\bG, F)$, respectively, and $H^\ast={\bH^\ast}^{F}$ and
$G^\ast={\bG^\ast}^F$.

Define $\bar H$ to be the group $\GO^\epsilon_{2n}(q)$ in the case
$S=\POmega^\epsilon_{2n}(q)$, and $\bar H := H$ otherwise.  We also
let $\Omega$ be the unique subgroup of index $2$ in $H$ for the
orthogonal cases when $q$ is odd, and let $\Omega=H$ otherwise, so that
$\Omega/\bZ(\Omega)=S=G/\bZ(G)$ and $\Omega\lhd \bar{H}$.  Note that
$B_0(S)$ can be identified with $B_0(\Omega)$ or with $B_0(G)$, by
\cite[Theorem 9.9]{Navarro}.

Now, let $e:=e_p(q)/\gcd(e_p(q), 2)$ and write $n=we+m$ with $0\leq m<e$. From Section \ref{sec:nonabelian},
we may again assume $w<p$. To obtain our result, we will rely on the case
of linear groups and will use some of the ideas of the arguments used in \cite[Propositions 5.4 and 5.5]{malle17},
which provides an analogue in this situation to the results of Michler and Olsson discussed above.
Namely, \cite[Propositions 5.4 and 5.5]{malle17} tells us
\[k(B_0(\bar H))=k\left(2e+\frac{p^a-1}{2e}, w\right),\] where $p^a=(q^{2e}-1)_p$.
Note that again, this number is at least $2\sqrt{p-1}$ (with strict inequality when $w\geq 2$), but that we wish to show the inequality for $k(B_0(A))$.  In most cases, we will again show that the number of $\Aut(S)$-orbits of characters in $B_0(S)$ is at least $2\sqrt{p-1}$.

If $w=1$, a Sylow $p$-subgroup of $\Omega$, $G$, $H$, or $\bar H$
(recall $p\geq 11$) is cyclic, so we may assume by Lemma
\ref{lemma-KS} that $w\geq 2$. Note that the unipotent characters of
$H$ are irreducible on restriction to $\Omega$. Arguing as in the
first paragraph of the proof of \cite[Lemma 3.10]{RSV21}, if $S\neq
D_4(q)$ nor $\Sp_4(2^{f})$ with $f$ odd, then the number of
$\Aut(S)$-orbits of unipotent characters in $B_0({H})$, and hence
$B_0(S)$, is $k(2e, w)$.  Note that $k(2e, w)> 4e$ since $w\geq 2$.

The characters in $B_0(H)$ and $B_0(G)$ lie in Lusztig series
indexed by $p$-elements $t$ of $H^\ast$, respectively $G^\ast$, by
\cite[Theorem 9.12]{Cabanes-book}. Note that centralizers of
odd-order elements of $\bH^\ast$ and of $\bG^\ast$ are always
connected (see e.g. \cite[Exercise (20.16)]{malletesterman}) and
that every odd $p$ is good for $\bH$ and $\bG$, so that $\chi_{(t)}$
lies in $B_0(H)$, respectively $B_0(G)$, for every $p$-element $t$
of $H^\ast$, respectively $G^\ast$, by Theorem \ref{lemma-red-gp}.
Further, note that the action on $\chi_{(t)}$ under a graph-field
automorphism of $H$ is determined by the action of a corresponding
graph-field automorphism on $(t)$, by \cite[Corollary
2.8]{Navarro-Tiep-Turull08}. (See also \eqref{eq:semisimple} below.)

Now let $\bG\hookrightarrow\wt{\bG}$ be a regular embedding as in
\cite[15.1]{Cabanes-book} and let $\wt{G}:=\wt{\bG}^F$.  Then the
action of $\wt{G}$ on $G$ induces all diagonal automorphisms of $S$.
Now, since $C_{\bG^\ast}(t)$ is connected for any $p$-element $t\in
G^\ast$, we have every character in $\mathcal{E}(G, (t))$ extends to
a character in $\wt{G}$.  (Indeed, since $\wt{G}/G$ is abelian and
restrictions from $\wt{G}$ to $G$ are multiplicity-free, the number
of characters lying below a given $\wt{\chi}\in\Irr(\wt{G})$ is the
number of $\beta\in\Irr(\wt{G}/G)$ such that
$\wt{\chi}\beta=\wt{\chi}$, as noted in \cite[Lemma 1.4]{RSV21}.
Hence \cite[Corollary 2.8]{bonnafe} and \cite[Proposition
2.6]{SFT21} yields the claim.) Therefore, each member of $B_0(S)$ is
invariant under diagonal automorphisms.

First consider the case $H=\SO_{2n+1}(q)$ or $\Sp_{2n}(q)$, so
$H^\ast=\Sp_{2n}(q)$ or $\SO_{2n+1}(q)$, respectively. Note that
$\Aut(S)/S$ in this case is generated by field automorphisms, which
also act on $H$, along with a diagonal or graph automorphism of
order at most $2$.

If $H=\SO_{2n+1}(q)$, then $\GL_n(q)$ may be embedded into
$H^\ast=\Sp_{2n}(q)$ in a natural way (namely, block diagonally as
the set of matrices of the form $(A, A^{-T})$ for $A\in\GL_n(q)$),
and the conjugacy class of $t$ is again determined by its
eigenvalues. Arguing as in the case of $\SL_n(q)$ above and noting
that every eigenvalue of $t$ must have the same multiplicity as its
inverse, we then have at least $\frac{p^a-1}{4e}$ distinct orbits of
non-unipotent characters in $B_0(H)$ under the field automorphisms,
and hence at least  $\frac{p^a-1}{4e}$  orbits in $B_0(S)$ under
$\Aut(S)$.  This gives more than $4e+\frac{p^a-1}{4e}$ orbits in
$\Irr(B_0(S))$ under $\Aut(S)$, which proves Theorem
\ref{main-thm-simple} in this case using \eqref{eq:step1}.

If $H=\Sp_{2n}(q)$, by \cite[Theorem 4.2]{geckhiss91}, there is a
bijection between classes of $p$-elements of $H$ and $H^\ast$, and
we note that field automorphisms act analogously on the $p$-elements
of $H$ and $H^\ast$.  Then the above again yields the result in this
case as long as $S\neq \Sp_4(2^f)$ with $f$ odd.

If $S= \Sp_4(2^f)$ with $f$ odd, then we must have $e=1$ and $w=2$.  Here \cite[Theorem 2.5]{Malle08} tells us that there is a
pair of unipotent characters permuted by the exceptional graph
automorphism, leaving $k(2,2)-1=4$
orbits of unipotent characters in $B_0(S)$ under $\Aut(S)$.  In this case, arguing as before and considering the action of the graph automorphism gives at least $4+\frac{p^a-1}{8}$ orbits in $B_0(S)$ under $\Aut(S)$, which is at least $2(p-1)^{1/4}$.  Hence part (ii) of Theorem \ref{main-thm-simple} holds.  So let $S\leq A\leq\Aut(S)$, and we wish to show that $B_0(A)$ contains more than $2\sqrt{p-1}$ characters. Note that in this case, $\Aut(S)/S$ is cyclic.  Let $X:=SC_A(P)$ for $P\in\Syl_p(S)$. Then $A/X$ is cyclic, say of size $b$, and $B_0(A)$ is the unique block covering $B_0(X)$ by \cite[(9.19) and (9.20)]{Navarro}.  Note that since at least $3$ of the unipotent characters of $S$ are $A$-invariant, we have at least $3b$ characters in $B_0(A)$ lying above unipotent characters.  Further, since the automorphisms corresponding to those in $X$ stabilize $p$-classes in $G^\ast$, the arguments above give at least $\frac{1}{2}\cdot\left(\frac{p^a-1}{2}\right)^2$ members of $B_0(X)$ lying above semisimple characters of $S$, and hence there are at least $\frac{(p^a-1)^2}{8b}$ members of $B_0(A)$ lying above semisimple characters of $S$.  Note then that the size of $B_0(A)$ is at least $3b+\frac{(p^a-1)^2}{8b}$, which is larger than $2\sqrt{p-1}$, completing the proof in this case.

Now, suppose we are in the case that $\bar{H}=\GO_{2n}^\epsilon(q)$. Note that the action of $\bar{H}/H$ induces a graph automorphism of order 2 in the case $\epsilon=1$, and that $\Aut(S)/S$ is generated by a group of diagonal automorphisms of size at most 4, along with graph and field automorphisms. Further, note that the action of $H$ on $\Omega$ induces a diagonal automorphism of order $2$ on $S$.
We may embed $\bar{H}$ in $\SO_{2n+1}(q)$, and by \cite[Proof of Proposition 5.5]{malle17}, the classes of $p$-elements $t$ with Lusztig series
contributing to $B_0(\bar H)$ are parametrised exactly as in the case of $\SO_{2n+1}(q)$ above.

Assume that  $(n,\epsilon)\neq (4,1)$. By again considering semisimple characters $\chi_{(t)}$ of $H$ with $t\in H^\ast$ $p$-elements, we may conclude that the number of orbits of non-unipotent characters in $B_0(S)$ under the $\Aut(S)$  is at least $\frac{p^a-1}{4e}$.  This yields at least $k(2e,w)+\frac{p^a-1}{4e}$ orbits in $\Irr(B_0(S))$ under $\Aut(S)$.    Hence we have the number of $\Aut(S)$ orbits in $B_0(S)$ is strictly larger than $4e+\frac{p^a-1}{4e}$, completing Theorem \ref{main-thm-simple} again in this case using \eqref{eq:step1}.

Finally, suppose $S=D_4(q)=\POmega_8^+(q)$ so $\bar H=\GO_8^+(q)$.
In this case, the graph automorphisms generate a group of size 6,
and a triality graph automorphism of order 3 permutes two triples of
unipotent characters (see \cite[Theorem 2.5]{Malle08}).
Since $w\geq 2$, we have $(e, w)\in \{(1,4), (2,2)\}$. The arguments
above give at least $k(2e, w)-4+\frac{p^a-1}{12e}$ distinct
$\Aut(S)$-orbits in $\Irr(B_0(S))$.  Since $k(2,4)=20$, by again
applying \eqref{eq:step1}, we may assume $e=2=w$.  In this case, we
have $k(2e,
w)-4+\frac{p^a-1}{12e}=10+\frac{p^a-1}{24}>2(p-1)^{1/4}$, so Theorem
\ref{main-thm-simple}(ii) is proved in this case. 

Now, let $S\leq A\leq \Aut(S)$, let $\Gamma$ be the subgroup of $\Aut(S)$ generated by inner, diagonal, and graph automorphisms, and let $X:=(\Gamma\cap A)C_A(P)$.  Then $A/X$ is cyclic, and by \cite[(9.19) and (9.20)]{Navarro}, $B_0(A)$ is the unique block covering $B_0(X)$. Let $b:=|A/X|$.
Now, the arguments above give at least $\frac{1}{3}\cdot\frac{1}{2}\cdot\left(\frac{p-1}{4}\right)^2$ members of $B_0(X)$ lying above semisimple characters of $S$, since members of $C_A(P)$ correspond to automorphisms stabilizing classes of $p$-elements of $G^\ast$, and hence there are at least $\frac{(p-1)^2}{96b}$ members of $B_0(A)$ lying above semisimple characters of $S$.  Further, there are at least 10 characters in $B_0(X)$ lying above unipotent characters in $B_0(S)$.  Since unipotent characters extend to their inertia groups and are invariant under field automorphisms (see \cite[Theorems 2.4 and 2.5]{Malle08}), this gives at least $10b$ elements of $B_0(A)$ lying above unipotent characters of $S$.  Together, this gives $k(B_0(A))\geq 10b+\frac{(p-1)^2}{96b}>2\sqrt{p-1}$, proving part (i) of Theorem \ref{main-thm-simple}.
\end{proof}


\section{A general bound for the number of $\Aut(S)$-orbits on
$\Irr(B_0(S))$}\label{sec:gen-bound}

The aim of this section is to obtain a general bound for the number
of $\Aut(S)$-orbits on irreducible ordinary characters in the
principal block of $S$, for $S$ a simple group of Lie type.

Building off of Theorem \ref{lemma-red-gp}, we will show that the principal block of $S$ contains \emph{many}
irreducible semisimple characters. By controlling the length of
$\Aut(S)$-orbits on these characters, we are able to bound below the
number of $\Aut(S)$-orbits on $\Irr(B_0(S))$. The bound turns out to
be enough to prove Theorem~\ref{main-thm-simple} for groups of
exceptional types, at least in the case when the Sylow $p$-subgroups
of the group of inner and diagonal automorphisms of $S$ are abelian
but non-cyclic, which is precisely the case we need after
Sections~\ref{sec:AM} and \ref{sec:nonabelian}.

\subsection{Specific setup for our purpose}

From now on we will work with the following setup: $\bG$ is a simple
algebraic group of adjoint type defined over $\FF_q$ and $F$ a
Frobenius endomorphism on $\bG$ such that $S=[\GG,\GG]$ with
$\GG=\bG^F$. Let $(\bG^\ast,F^\ast)$ be the dual pair of $(\bG,F)$
and for simplicity we will use the same notation $F$ for $F^\ast$,
and thus $\bG^\ast$ is a simple algebraic group of simply connected
type and $S=\GG^\ast/\bZ(\GG^\ast)$, where $\GG^\ast:=(\bG^\ast)^F$.

Theorem~\ref{lemma-red-gp} has the following consequence.

\begin{lemma}\label{lemma:semisimple-char-in-B0} Assume the above notation.
Let $p$ be a good prime for $\bG$
 not dividing $q$. For every $p$-element $t$ of ${\GG^\ast}$, the
semisimple character $\chi_{(t)}\in \mathcal{E}(\GG,(t))$ belongs to
the principal block of $\GG$.
\end{lemma}

\begin{proof}
Since $\bG^\ast$ is of simply connected type, the centralizer
$\bC_{\bG^\ast}(t)$ is connected for every semisimple element
$t\in\GG^\ast$, by \cite[Exercise 20.16]{malletesterman}. The lemma
follows from Theorem~\ref{lemma-red-gp}.
\end{proof}

\subsection{Orbits of semisimple characters} Knowing that the semisimple characters
$\chi_{(t)}\in\Irr(\GG)$ associated to $\GG^\ast$-conjugacy classes
of $p$-elements all belong to $B_0(\GG)$, we now wish to control the
number of orbits of the action of the automorphism group $\Aut(S)$
on these characters. By a result of Bonnaf\'{e}
\cite[\S2]{Navarro-Tiep-Turull08}, this action turns out to be
well-behaved.

Let $\alpha\in\Aut(\GG)$, which in our situation will be a product
of a field automorphism and a graph automorphism. It is easy to see
that $\alpha$ then can be extended to a bijective morphism
${\overline{\alpha}}: \bG\rightarrow \bG$ such that
${\overline{\alpha}}$ commutes with $F$. This ${\overline{\alpha}}$
induces a bijective morphism
$\overline{\alpha}^\ast:\bG^\ast\rightarrow \bG^\ast$ which commutes
with the dual of $F$. The restriction of $\overline{\alpha}^\ast$ to
$\GG^\ast$, which we denote by $\alpha^\ast$, is now an automorphism
of $\GG^\ast$. Recall that $\alpha\in\Aut(\GG)$ induces a natural
action on $\Irr(\GG)$ by $\chi^\alpha:=\chi\circ \alpha^{-1}$. By
\cite[\S2]{Navarro-Tiep-Turull08}, $\alpha$ maps the Lusztig series
$\mathcal{E}(\GG,(t))$ of $\GG$ associated to $(t)$ to the series
$\mathcal{E}(\GG,(\alpha^\ast(t)))$ associated to
$(\alpha^\ast(t))$. Consequently, if $\bC_{\bG^\ast}(t)$ is
connected, then $\bC_{\bG^\ast}(\alpha^\ast(t))$ is also connected,
and \begin{equation}\label{eq:semisimple}
{\chi_{(t)}}^\alpha=\chi_{(\alpha^\ast(t))},
\end{equation}
which means that an automorphism of $\GG$ maps a semisimple
character associated to a conjugacy class $(t)$ (of $\GG^\ast$) to
the semisimple character associated to $(\alpha^\ast(t))$.

Due to Section \ref{sec:nonabelian} and Subsection~\ref{sec:AM}, we
may assume that the Sylow $p$-subgroups of $\GG$ are abelian but not
cyclic. Therefore, $\GG$ is not of type $\ta B_2$ or $\ta G_2$.
Assume for a moment that $\GG$ is not of type $\ta F_4$ as well.
Then there is a unique positive integer $e$ such that $p\mid
\Phi_e(q)$ and $\Phi_e$ divides the generic order of $\GG$. (Recall
that $\Phi_e$ denotes the $e$th cyclotomic polynomial.) This $e$
then must be the multiplicative order of $q$ modulo $p$, which means
that $p\mid (q^e-1)$ but $p\nmid (q^i-1)$ for every $0<i<e$. In the
case $\GG$ is of type $\ta F_4$, we use $\Phi_{4^{\pm}}(q):=q\pm
\sqrt{2q}+1$, and what we discuss below still holds with slight
modification.

Let $\Phi_e(q)=p^am$ where $\gcd(p,m)=1$ and $\Phi_e^{k_e}$ the
precise power of $\Phi_e$ dividing the generic order of $\GG$. We
will use $k$ for $k_e$ for convenience if $e$ is not specified. A
Sylow $e$-torus of $\GG^\ast$ has order $\Phi_e(q)^k$ and contains a
Sylow $p$-subgroup of $\GG^\ast$. Sylow $p$-subgroups of $\GG^\ast$
(and $\GG$) are then isomorphic to
\[
\underbrace{C_{p^a}\times C_{p^a}\times\cdots \times C_{p^a}}_{k
\text{ times}}.
\]

Assume that $q=\ell^f$ where $\ell$ is the defining characteristic
of $S$.

\begin{lemma}\label{lemma-field-auto} Assume the above notation. Let $\alpha$ be a field
automorphism of $\GG$ of order $f$, and thus $\langle \alpha\rangle$
is the group of field automorphisms of $\GG$. Each $\alpha$-orbit on
semisimple characters $\chi_{(t)}\in\Irr(\GG)$ associated to
conjugacy classes of $p$-elements ($p\neq \ell$) has length at most
$\min\{f,p^a-p^{a-1}\}$.
\end{lemma}

\begin{proof} Let $\alpha^\ast$ be an automorphism of $\GG^\ast$
constructed from $\alpha$ by the process described above. For
simplicity we use $\alpha$ for $\alpha^\ast$. By
\eqref{eq:semisimple} and since $|\alpha|=f$, it is enough to show
that each $\alpha$-orbit on $\GG^\ast$-conjugacy classes of
(semisimple) $p$-elements of $\GG^\ast$ is at most $p^a-p^{a-1}$.

Let $t\in\GG^\ast$ be a $p$-element. Note that each element in
$\GG^\ast$ conjugate to $t$ under $\bG^\ast$ is automatically
conjugate to $t$ under $\GG^\ast$, by \cite[(3.25)]{Digne-Michel91}
and the fact that $\bC_{\bG^\ast}(t)$ is connected. Let $t$ be
conjugate to $h_{\alpha_1}(\lambda_1)\cdots
h_{\alpha_n}(\lambda_n)$, where the $h_{\alpha_i}$ are the coroots
corresponding to a set of fundamental roots with respect to a
maximal torus $\bT^\ast$ of $\bG^\ast$ and $n$ is the rank of
$\bG^\ast$.  Since $\bG^\ast$ is simply connected, note that $(t_1,
\ldots, t_n)\mapsto h_{\alpha_1}(t_1)\cdots h_{\alpha_n}(t_n)$ is an
isomorphism from $(\overline{\FF}_q^\times)^n$ to $\bT^\ast$ (see
\cite[Theorem 1.12.5]{Gorensteinetal}).


Now, if $\lambda=\lambda_i$ for some $1\leq i\leq n$, then
$\lambda^{p^a}=1$, since $|t|\mid p^a$. Recall that $\ell\neq p$,
and thus $\ell^{p^a-p^{a-1}}\equiv 1(\bmod \,p^a)$ by Euler's
totient theorem. It follows that
$\lambda^{\ell^{p^a-p^{a-1}}}=\lambda$, which yields that the
$\alpha$-orbit on $(t)$ is contained in
$\{(t),(\alpha(t)),...,(\alpha^{p^a-p^{a-1}-1}(t))\}$, as desired.
\end{proof}

\subsection{A bound for the number of $\Aut(S)$-orbits on
$\Irr(B_0(S))$}\label{subsec:a-bound}

Let $\bT_e$ be an $F$-stable maximal torus containing a Sylow
$e$-torus of $\bG^\ast$, and thus $\bT_e$ contains a Sylow
$p$-subgroup $P$ of $\GG^\ast$. Let
$W(\bT_e):=\bN_{\GG^\ast}(\bT_e)/\bT_e^F$ be the relative Weyl group
of $\bT_e$. It is well-known that fusion of semisimple elements in a
maximal torus is controlled by its relative Weyl group (see
\cite[Exercise 20.12]{malletesterman} or \cite[p. 6]{Malle-Maroti}).
Therefore, the number of conjugacy classes of (nontrivial)
$p$-elements of $\GG^\ast$ is at least
\[
\frac{|P|-1}{|W(\bT_e)|}=\frac{p^{ak}-1}{|W(\bT_e)|}.
\]
Note that $\chi_{(t)}$ belongs to the Lusztig series
$\mathcal{E}(\GG,(t))$ defined by the conjugacy class $(t)$ and the
Lusztig series are disjoint, and so two semisimple characters
$\chi_{(t)}$ and $\chi_{(t_1)}$ are equal if and only if $t$ and
$t_1$ are conjugate in $\GG^\ast$. Therefore, using
Lemma~\ref{lemma:semisimple-char-in-B0}, we deduce that, when $p$ is
a good prime for $\bG$ and not dividing $q$,
\begin{equation}\label{eq:Irr-ss-B0}
|\Irr_{ss}(B_0(\GG))|\geq \frac{p^{ak}-1}{|W(\bT_e)|},
\end{equation}
where $\Irr_{ss}(B_0(\GG))$ denotes the set of (nontrivial)
semisimple characters (associated to $p$-elements of $\GG^\ast$) in
$B_0(\GG)$. Let $n(X,Y)$ denote the number of $X$-orbits on a set
$Y$. Using Lemma~\ref{lemma-field-auto}, we then have
\begin{align*}
n(\Aut(S),\Irr_{ss}(B_0(\GG)))&\geq
\frac{p^{ak}-1}{g\min\{f,p^a-p^{a-1}\}|W(\bT_e)|}\\
&\geq\frac{p^{k}-1}{g(p-1)|W(\bT_e)|},
\end{align*}
where $g$ is the order of the group of graph automorphisms of $S$.
Let $d:=|\GG/S|$ -- the order of the group of diagonal automorphisms
of $S$ and viewing the irreducible constituents of the restrictions
of semisimple characters of $\GG$ to $S$ as semisimple characters of
$S$, we now have
\begin{equation}\label{equ:ss-orbit-bound}
n(\Aut(S),\Irr_{ss}(B_0(S)))\geq \frac{p^{k}-1}{dg(p-1)|W(\bT_e)|}.
\end{equation}

We note that values of $d,f$, and $g$ for various families of simple
groups are known, see \cite[p. xvi]{Atl1} for instance.

We now turn to unipotent characters in the principal block $B_0(S)$.
Brou\'{e}, Malle, and Michel \cite[Theorem 3.2]{Broue-Malle-Michel}
partitioned the set $\mathcal{E}(\GG^\ast,1)$ of unipotent
characters of $\GG^\ast$ into $e$-Harish-Chandra series associated
to $e$-cuspidal pairs of $\bG^\ast$, and furthermore obtained
one-to-one correspondences between $e$-Harish-Chandra series and the
irreducible characters of the relative Weyl groups of the
$e$-cuspidal pairs defining these series. Cabanes and Enguehard
\cite[Theorem 21.7]{Cabanes-book} then proved the compatibility
between Brou\'{e}-Malle-Michel's partition of unipotent characters
of $\GG^\ast$ by $e$-Harish-Chandra series and the partition of
unipotent characters by unipotent blocks. These results imply that
the number of unipotent characters in $B_0(S)$ (and $B_0(\GG^\ast)$
as well) is the same as the number of conjugacy classes of the
relative Weyl group $W(\bL_e)$ of the centralizer
$\bL_e:=\bC_{\bG^\ast}(\bS_e)$ of a Sylow $e$-torus $\bS_e$ of
$\bG^\ast$. Here we note that $\bL_e$ is a minimal $e$-split Levi
subgroup of $\bG^\ast$, and $\bL_e=\bS_e$ if $\bS_e$ happens to be a
maximal torus of $\bG^\ast$, since every maximal torus is equal to
its centralizer in a connected reductive group.

By the aforementioned result of Lusztig (\cite[Theorem
2.5]{Malle08}), every unipotent character of a simple group of Lie
type lies in a $\Aut(S)$-orbit of length at most 3. In fact, every
unipotent character of $S$ is $\Aut(S)$-invariant, except in the
following cases:
\begin{enumerate}
\item $S=P\Omega_{2n}^+(q)$ ($n$ even), the graph automorphism of order 2
has $o_2(S)$ orbits of length 2, where $o_2(S)$ is the number of
degenerate symbols of defect 0 and rank $n$ parameterizing unipotent
characters of $S$ (see \cite[p. 471]{Carter85}).

\item $S = P\Omega^+_8(q)$, the graph automorphism of order 3 has
$o_3(S)=2$ orbits of length 3, each of which contains one pair of
characters parameterized by one degenerate symbol of defect 0 and
rank $2$ in (1).

\item $S = \Sp_4(2^f)$ ($f$ odd), the graph automorphism of order 2
has $o_2(S)=1$ orbit of length 2.

\item $S = G_2(3^f)$ ($f$ odd), the graph automorphism of order 2 has
$o_2(S)=1$ orbit of length 2 on unipotent characters.

\item $S= F_4(2^f)$ ($f$ odd), the graph automorphism of order
2 has $o_2(S)=8$ orbits of length 2 on unipotent characters.
\end{enumerate}

Combining this with the bound \eqref{equ:ss-orbit-bound}, we obtain:

\begin{theorem}\label{theorem-bound}
Let $S$ be a simple group of Lie type. Let $p$ be a good prime for
$S$ and different from the defining characteristic of $S$. Assume
that Sylow $p$-subgroups of the group of inner and diagonal
automorphisms of $S$ are abelian. Let $k,d,f,g$, $\bT_e$, and
$\bL_e$ as above. Let $n(S)$ denote the number of $\Aut(S)$-orbits
on irreducible ordinary characters in $B_0(S)$. Then
\begin{align*}
n(S) &\geq k(W(\bL_e))+ \frac{p^{k}-1}{dg(p-1)|W(\bT_e)|},
\end{align*}
except possibly the above cases $(1), (3), (4)$, and $(5)$ in which
the bound is lower by the number $o_2(S)$ of orbits of length $2$ on
unipotent characters and case $(2)$ in which the bound is lower by
$4$.
\end{theorem}

We remark that when $e$ is regular for $\bG^\ast$, which means that
the centralizer $\bC_{\bG^\ast}(\bS_e)$ of the Sylow $e$-torus
$\bS_e$ is a maximal torus of $\bG^\ast$, the maximal torus $\bT_e$
containing $\bS_e$ can be chosen to be precisely
$\bL_e=\bC_{\bG^\ast}(\bS_e)$. When $e$ is not regular, $W(\bT_e)$
is always bigger than $W(\bL_e)$. However, for exceptional types,
$e$ being non-regular happens only when $\bG$ is of type $E_7$ and
$e\in\{4, 5, 8, 10, 12\}$, see \cite[Table 3]{Broue-Malle-Michel}.
We thank G. Malle for pointing out these facts to us.

We also remark that when the Sylow $p$-subgroups of the group of
inner and diagonal automorphisms of $S$ are furthermore noncyclic,
then $k\geq 2$, and, away from those exceptions, we have a rougher
bound
\begin{equation}\label{bound-rougher} n(S)\geq
k(W(\bL_e))+ \frac{p+1}{dg|W(\bT_e)|},
\end{equation} but turns out to be sufficient for our purpose in most cases.


\section{Groups of exceptional types}\label{sec:proof-1.2-except}

In this section we prove Theorem \ref{main-thm-simple} for $S$ being
of exceptional type. This is achieved by considering each type case
by case, with the help of Theorem~\ref{theorem-bound}.

We keep all the notation in Section \ref{sec:gen-bound}. In
particular, the underlying field of $S$ has order $q=\ell^f$. By
Section \ref{sec:AM}, we may assume that $\ell\neq p\geq 11$. This
assumption on $p$ guarantees that Sylow $p$-subgroups of $\GG$ are
abelian. Recall also that $e$ is the multiplicative order of $q$
modulo $p$, $p^a=\Phi_e(q)_{p}$, and $\Phi_e^k=\Phi_e^{k_e}$ is the
precise power of $\Phi_e$ dividing the generic order of $\GG$. By
Section \ref{sec:AM}, we may assume that the Sylow $p$-subgroups of
$S$ are not cyclic, and thus $k_e\geq 2$. Also, $\bS_e$ is a Sylow
$e$-torus of a simple algebraic group $\bG^\ast$ of simply connected
type associated with a Frobenius endomorphism $F$ such that
$S=\GG^\ast/\bZ(\GG^\ast)$ and $\GG^\ast:={\bG^\ast}^F$, and $\bT_e$
is an $F$-stable maximal torus of $\bG^\ast$ containing $\bS_e$. As
mentioned already, we choose $\bT_e=\bL_e:=\bC_{\GG^\ast}(\bS_e)$ if
$e$ is regular for $\bG^\ast$. This indeed is the case for all types
and all $e$, except the single case of type $E_7$ and $e=4$. The
relative Weyl groups $W(\bL_e)$ are always finite complex reflection
groups, and we will follow the notation for these groups in
\cite{Benard76}. Relative Weyl groups for various $\bL_e$ are
available in \cite[Tables 1 and 3]{Broue-Malle-Michel}. The
structure of $\Out(S)$ is available in \cite[Theorem
2.5.12]{Gorensteinetal}. We will use these data freely without
further notice.

It turns out that Theorem \ref{theorem-bound} is sufficient to prove
Theorem~\ref{main-thm-simple} whenever $k_e\geq 3$. In fact, even
when $k_e=2$, Theorem~\ref{theorem-bound} is also sufficient for
Theorem~\ref{main-thm-simple}(ii). We have to work harder, though,
to achieve Theorem~\ref{main-thm-simple}(i) in the case $k_e=2$ for
some types.

\begin{proposition}
Theorem \ref{main-thm-simple} holds for simple groups of exceptional
types.
\end{proposition}

\begin{proof}

\textbf{(1) {$S=G_2(q)$ and $S=F_4(q):$}}

\smallskip

First we consider $S=G_2(q)$ (so $S=\GG$) with $q>2$. Then
$e\in\{1,2\}$ and $k_1=k_2=2$. Also, the Sylow $e$-tori are maximal
tori, and their relative Weyl groups are the dihedral group
$D_{12}$. The bound \eqref{bound-rougher} implies that $n(S)\geq
5+(p+1)/12$ for $q=3^f$ with odd $f$, and $n(S)\geq 6+(p+1)/12$
otherwise. In any case it follows that $n(S)>2(p-1)^{1/4}$, proving
Theorem ~\ref{main-thm-simple}(ii) for $G_2(q)$.

Note that $\Aut(S)$ is a cyclic extension of $S$. First assume that
$q\neq 3^f$ with odd $f$ or $G$ does not contain the graph
automorphism of $S$. In particular, every unipotent character of $S$
is extendible to $G$. Let $H:=\langle S,\bC_G(P)\rangle$, where $P$
is a Sylow $p$-subgroup of $G$ (and $S$ as well by the assumption
$p\nmid |G/S|$). Since $P\bC_G(P)$ is contained in $H$, $B_0(H)$ is
covered by a unique block of $G$, which is $B_0(G)$. It follows
that, each unipotent character in $B_0(S)$ extends to an irreducible
character in $B_0(H)$, which in turns lies under $|G/H|$ irreducible
characters in $B_0(G)$. Therefore, the number of irreducible
characters in $B_0(G)$ lying over unipotent characters of $S$ is at
least $k(D_{12})|G/H|=6|G/H|$. When $q=3^f$ with odd $f$ and $G$
does contain the nontrivial graph automorphism, similar arguments
yield that the number of irreducible characters in $B_0(G)$ lying
over unipotent characters of $S$ is at least $5|G/H|$.

On the other hand, each $G$-orbit on semisimple characters
(associated to $p$-elements) of $S$ now has length at most $|G/H|$
by \eqref{eq:semisimple} and the fact that $H=\langle
S,\bC_G(P)\rangle$ fixes every conjugacy class of $p$-elements of
$S$. Therefore, the bound \eqref{eq:Irr-ss-B0} yields
\[
n(G,\Irr_{ss}(B_0(S)))\geq \frac{p^{2}-1}{12|G/H|}.
\]
This and the conclusion of the last paragraph imply that
\[
k(B_0(G))\geq 5|G/H|+\frac{p^{2}-1}{12|G/H|}\geq
2\sqrt{\frac{5(p^2-1)}{12}},
\]
which in turns implies the desired bound $k(B_0(G))> 2 \sqrt{p-1}$
for all $p\geq 11$.

For $S=F_4(q)$, we have $e\in\{1,2\}$ for which $k_e=4$, or
$e\in\{3,4,6\}$ for which $k_e=2$. Therefore all the Sylow $e$-tori
are maximal tori, and their relative Weyl groups are
$G_{28}=\GO_4^+(3)$ for $e=1,2$; $G_5=\SL_2(3)\times C_3$ for
$e=3,6$; and $G_8=C_4.\Sy_4$ for $e=4$. Now we just follow along
similar arguments as above to prove the theorem for this type.

\medskip

\textbf{(2) {$S=\ta F_4(q)$ with $q=2^{2n+1}\geq 8$ and $S=\tb
D_4(q)$:}}

\smallskip

These two types are treated in a fairly similar way as for $G_2$.
Note that $\Out(S)$ here is always cyclic. First let $S=\ta F_4(q)$.
Then $e\in\{1,2,4^+,4^-\}$ and $k_e=2$ for all $e$. All the Sylow
$e$-tori are maximal. The relative Weyl groups of these tori are
$D_{16}$, $G_{12}=\GL_2(3)$, $G_8=C_4.\Sy_4$ and $G_8$ for
$e=1,2,4^+$, and $4^-$, respectively. One can now easily check the
inequality $n(S)\geq 2(p-1)^{1/4}$, using (\ref{bound-rougher}). The
bound $k(B_0(G))> 2 \sqrt{p-1}$ is proved similarly as in type
$G_2$.

Now let $S=\tb D_4(q)$. Then $e\in\{1,2,3,6\}$ and $k_e=2$ for all
$e$. For $e\in \{3,6\}$, a Sylow $e$-torus is maximal with the
relative Weyl group $G_4=\SL_2(3)$. For $e=1$ or $2$, Sylow $e$-tori
of $S$ are not maximal anymore but are contained in maximal tori of
orders $\Phi_1^2(q)\Phi_3(q)$ and $\Phi_2^2(q)\Phi_6(q)$,
respectively. The relative Weyl groups of these tori are both
isomorphic to $D_{12}$. Now the routine estimates are applied to
achieve the required bounds.

\medskip

\textbf{(3) $S=E_6(q)$ and $S=\ta E_6(q)$:}

\smallskip

These two types are approached similarly and we will provide details
only for $E_6$. Then $e=1$ for which $k_e=6$, or $e=2$ for which
$k_e=4$, or $e=3$ for which $k_e=3$, or $e\in\{4,6\}$ for which
$k_e=2$.

Assume $e=1$. Then $\bS_1$ is a maximal torus and its Weyl group is
$G_{35}=\SO_5(3)$. Theorem~\ref{theorem-bound} then implies that
\[n(S)\geq
k(\SO_5(3))+\frac{p^{6}-1}{6(p-1)|\SO_5(3)|}=25+\frac{p^{6}-1}{311040(p-1)}>2\sqrt{p-1},\]
proving both parts of Theorem \ref{main-thm-simple} in this case.
The case $e\in\{2,3\}$ is similar. We note that $\bS_3$ is a maximal
torus with the relative Weyl group $G_{25}=3^{1+2}.\SL_2(3)$, 
and a maximal torus containing a Sylow $2$-torus has relative Weyl
group $G_{28}$.

Assume $e=4$. Then a maximal torus containing a Sylow $4$-torus of
$E_6(q)_{sc}$ has order $\Phi_4^2(q)\Phi_1^2(q)$ and its relative
Weyl group is $G_8=C_4.\Sy_4$, whose order is 96 and class number is
16. Now the bound \eqref{bound-rougher} yields $n(S)>2(p-1)^{1/4}$,
proving part (ii) of the theorem.

We need to to do more to obtain part (i) in this case. In fact, when
$2\sqrt{p-1}\leq 16$, which means that $p\leq 65$, we have $n(S)>
16\geq 2\sqrt{p-1}$, which proves part (i) as well. So let us assume
that $p>65$.

Note that $\Out(S)$ is a semidirect product $C_{(3,q-1)}\rtimes
(C_f\times C_2)$, which may not be abelian but every unipotent
character of $S$ is still fully extendible to $\Aut(S)$ by
\cite[Theorems 2.4 and 2.5]{Malle08}. As before, let $\GG$ be the
extension of $S$ by diagonal automorphisms. Similar to the proof for
type $G_2$, let $H:=\langle G\cap \GG,\bC_G(P)\rangle$, where $P$ is
a Sylow $p$-subgroup of $S$. Each unipotent character in $B_0(S)$
then lies under at least $|\Irr(G/H)|=|G/H|$ irreducible characters
in $B_0(G)$. (Here we note that $G/H$ is abelian.) Thus, the number
of irreducible characters in $B_0(G)$ lying over unipotent
characters of $S$ is at least $16|G/H|$.

As in Subsection~\ref{subsec:a-bound}, here we have
\[|\Irr_{ss}(B_0(\GG))|\geq \frac{p^2-1}{|W(\bT_4)|}=\frac{p^2-1}{96}.\]
Let $\Irr_{ss}(B_0(S))$ be the set of restrictions of characters in
$\Irr_{ss}(B_0(\GG))$ to $S$. Note that these restrictions are
irreducible as the sesisimple elements of $\GG^\ast$ associated to
these sesisimple characters are $p$-elements whose orders are
coprime to $|\bZ(\GG^\ast)|=\gcd(3,q-1)$. Also note that, if
 the restrictions of $\chi_{(t)}$ and $\chi_{(t_1)}$ to $S$ are the same, then $(t)=(t_1z)$
for some $z\in \bZ(\GG^\ast)$ (see \cite[Proposition 5.1]{Tiep15}),
which happens only when $z$ is trivial since $t$ and $t_1$ are
$p$-elements. It follows that
\[|\Irr_{ss}(B_0(S))|=|\Irr_{ss}(B_0(\GG))|\geq \frac{p^2-1}{96}.\]
Now, each $H$-orbit of conjugacy classes of $p$-elements of $S$ has
length at most $\gcd(3,q-1)\leq 3$. Therefore, each $G$-orbit of
semisimple characters in $B_0(S)$ has length at most $3|G/H|$, and
it follows that the number of irreducible characters in $B_0(G)$
lying over sesisimple characters in $B_0(S)$ is at least
${(p^2-1)}/({288|G/H|})$.
Together with the bound of $16|G/H|$ for the number of irreducible
characters in $B_0(G)$ lying over unipotent characters of $S$, we
deduce that
\[
k(B_0(G))\geq 16|G/H|+\frac{p^{2}-1}{288|G/H|}\geq
2\sqrt{\frac{16(p^2-1)}{288}},
\]
and thus, when $p>65$, the desired bound $k(B_0(G))> 2 \sqrt{p-1}$
follows.

The last case $e=6$ can be argued in a similar way, with notice that
a maximal torus containing a Sylow $6$-torus of $E_6(q)_{sc}$ has
order $\Phi_6^2(q)\Phi_3(q)$ and its relative Weyl group is
$G_5=\SL_2(3)\times C_3$, whose order is 72 and class number is 21.


\medskip

{\textbf{(4) $S=E_7(q)$:}}

\smallskip

Then $e\in\{1,2\}$ for which $k_e=7$, or $e\in\{3,6\}$ for which
$k_e=3$, or $e=4$ for which $k_e=2$. When $k_e>2$, the bound
\ref{bound-rougher} again is sufficient to achieve the desired bound
$n(S)> 2\sqrt{p-1}$. In fact, even for the case $k_e=2$, we have
$n(S)\geq 2(p-1)^{1/4}$. So it remains to prove
Theorem~\ref{main-thm-simple}(i) for $e=4$, in which case $e$ is not
regular and the relative Weyl group of the minimal $e$-split Levi
subgroup $\bL_e=\bS_e.A_1^3$ is $G_8$ (see
\cite[Table~1]{Broue-Malle-Michel}). As $A_1^3$ has eight classes of
maximal tori (two in each factor $A_1$), all of which have Weyl
group $C_2^3$, the relative Weyl group of every maximal torus
$\bT_e$ containing $\bS_e$ is of the form $G_8.C_2^3$. The estimates
are now similar to those in the case $e=4$ of the type $E_6$.

\medskip

{\textbf{(5) $S=E_8(q)$:}}

\smallskip

Then $e\in\{1,2\}$ for which $k_e=8$, or $e\in\{3,4,6\}$ for which
$k_e=4$, or $e\in\{5,8,10,12\}$ for which $k_e=2$. The standard
approach as above works for all $e$ with $k_e>2$.

Assume that $e\in\{5,10\}$. Then a Sylow $e$-torus of $S$ is maximal
and its relative Weyl group is $G_{16}\cong \SL_2(5)\times C_5$. A
similar proof to the case of type $G_2$ yields $k(B_0(G))\geq
2\sqrt{{45(p^2-1)}/{600}}$, which is certainly greater than
$2\sqrt{p-1}$ for $p\geq 13$. On the other hand, we always have
$k(B_0(G))\geq 45>2\sqrt{p-1}$ for smaller $p$, and thus the desired
bound holds for all $p$. Finally, the case $e\in\{8,12\}$ is
entirely similar, with notice that the relative Weyl groups of Sylow
$e$-tori are $G_9=C_8.\Sy_4$ and $G_{10}=C_{12}.\Sy_4$ for $e=8$ and
$12$, respectively.
\end{proof}

Theorem \ref{main-thm-simple} is now completely proved.


\section{Proof of Theorems \ref{main-theorem} and \ref{theorem:equality}}\label{sec:proof1.1}

We are now ready to prove the main results.

\begin{proof}[Proof of Theorems \ref{main-theorem} and \ref{theorem:equality}]
First we remark that the `if' implication of
Theorem~\ref{theorem:equality} is clear, and moreover, we are done
if the Sylow $p$-subgroups of $G$ are cyclic, thanks to
Subsection~\ref{sec:AM}.

Let $(G,p)$ be a counterexample to either Theorem~\ref{main-theorem}
or the `only if' implication of Theorem \ref{theorem:equality} with
$|G|$ minimal. Let $N$ be a minimal normal subgroup of $G$. In
particular, Sylow $p$-subgroups of $G$ are not cyclic and
$k(B_0(G))\leq 2\sqrt{p-1}$.

Assume first that $p\mid |G/N|$. Then, since
$\Irr(B_0(G/N))\subseteq \Irr(B_0(G))$ and by the minimality of
$|G|$, we have
\[2\sqrt{p-1}\geq k(B_0(G))\geq k(B_0(G/N))\geq 2\sqrt{p-1},\] and
thus \[ k(B_0(G))= k(B_0(G/N))= 2\sqrt{p-1}.
\] The minimality of $G$ again then implies that $G/N$ is isomorphic to the
Frobenius group $C_p\rtimes C_{\sqrt{p-1}}$. It follows that $p\mid
|N|$, and thus there exists a nontrivial irreducible character
$\theta\in \Irr(B_0(N))$. As $B_0(G)$ covers $B_0(N)$, there is some
$\chi\in \Irr(B_0(G))$ lying over $\theta$, implying that
$k(B_0(G))>k(B_0(G/N))$, a contradiction.

So we must have $p\nmid |G/N|$, and it follows that $p\mid |N|$.
This in fact also yields that $N$ is a unique minimal normal
subgroup of $G$. Assume first that $N$ is abelian. We then have that
$G$ is $p$-solvable, and hence Fong's theorem (see \cite[Theorem
10.20]{Navarro}) implies that
\[
k(B_0(G))=k(B_0(G/\pcore_{p'}(G)))=k(G/\pcore_{p'}(G)),
\]
which is greater than $2\sqrt{p-1}$ by the main result of
\cite{Maroti16}.

We now may assume that $N\cong S_1\times S_2\times \cdots \times
S_k$, a direct product of $k\in \NN$ copies of a non-abelian simple
group $S$. If $S$ has cyclic Sylow $p$-subgroups, then $G$ is not a
counterexample for Theorem~\ref{main-theorem} by
Lemma~\ref{lemma-KS}, and furthermore, Sylow $p$-subgroups of $G$
are abelian, implying that \[ k(B_0(G))\geq
k(\bN_G(P)/\bO_{p'}(\bN_G(P)))>2\sqrt{p-1}
\] by the analysis in Subsection~\ref{sec:AM}, \cite{Maroti16} again,
and the fact that $P\in\Syl_p(G)$ is not cyclic.

So the Sylow $p$-subgroups of $S$ are not cyclic. Let $n$ be the
number of $\bN_G(S_1)/N$-orbits on $\Irr(B_0(S_1))$. By
Theorem~\ref{main-thm-simple}(ii), we have $n\geq 2(p-1)^{1/4}$.
Therefore, if $k\geq 2$, the number of $G$-orbits on
$\Irr(B_0(N))=\prod_{i=1}^k \Irr(B_0(S_i))$ is at least
$n(n+1)/2\geq 2(p-1)^{1/4}(2(p-1)^{1/4}+1)/2>2\sqrt{p-1}$, and it
follows that $k(B_0(G))>2\sqrt{p-1}$, a contradiction. Hence, $N=S$
and $G$ is then an almost simple group with socle $S$. Furthermore,
$p\nmid |G/S|$. But such a group $G$ cannot be a counterexample by
Theorem~\ref{main-thm-simple}(i). The proof is complete.
\end{proof}

With Theorem \ref{main-theorem} in mind, it follows that for any
$p$-block $B$ for a finite group such that $B$ shares its invariants
with the the principal block of some finite group $H$ (or even just
satisfying $k(B)=k(B_0(H))$), we have $k(B)\geq 2\sqrt{p-1}$.  In
particular, we may record the following:

\begin{corollary}
Let $G$ be one of the classical groups $\GL_n(q)$, $\GU_n(q)$,
$\Sp_{2n}(q)$, $\SO_{2n+1}(q)$, or $\GO_{2n}^\pm(q)$.  Let $p$ be a prime dividing $|G|$ and not dividing $q$.  Then for any
$p$-block $B$ of $G$ with positive defect, we have $k(B)\geq 2\sqrt{p-1}$.
\end{corollary}

\begin{proof}
If $p=2$, then the statement is clear, so we assume $p$ is odd.  First, if $G=\GL_n(q)$ or $\GU_n(q)$, the statement follows immediately from
Theorem \ref{main-theorem} and \cite[Theorem (1.9)]{Michler-Olsson},
which states that $B$ has the same block invariants as the principal
block of a product of lower-rank general linear and unitary groups.

Now suppose that $G$ is $\Sp_{2n}(q)$, $\SO_{2n+1}(q)$, or
$\GO_{2n}^\pm(q)$.  If $B$ is a unipotent block, then by
\cite[Proposition 5.4 and 5.5]{malle17}, $B$ has the same block
invariants as an appropriate general linear group.  (In the case
$\GO_{2n}^\pm(q)$, we define a unipotent block to be one lying above
a unipotent block of $\SO_{2n}^\pm(q)$.)  Hence the statement holds
if $B$ is a unipotent block.

Now, the block $B$ determines a class of semisimple $p'$-elements
$(s)$ of the dual group $G^\ast$ (see \cite[Theorem
9.12]{Cabanes-book}) such that $B$ contains some member of
$\mathcal{E}(G, (s))$.  By \cite[Th{\'e}or{\`e}me 1.6]{enguehard08},
there exists a group $G(s)$ dual to $\bC_{G^\ast}(s)$ such that
$k(B)=k(b)$ for some unipotent block $b$ of $G(s)$.  Now, in the
cases under consideration, $\bC_{G^\ast}(s)$ and $G(s)$ are direct
products of lower-rank classical groups of the types being
considered here, completing the proof.
\end{proof}


\end{document}